\documentclass[a4paper,12pt,reqno]{amsart}
\usepackage{amsmath,amsfonts,amssymb,amsthm,enumerate,multicol}%hyperref,
\usepackage{mathrsfs}
\usepackage{tikz,graphicx}
\usepackage{hyperref}
\usepackage{caption,enumitem,wasysym}
\usepackage{cleveref}
\usepackage{epstopdf}
\usepackage{float,wrapfig}
\usepackage[labelsep=space]{caption}
\usepackage[font=footnotesize]{caption}
\usepackage{xcolor}
\usepackage[english]{babel}
\usepackage[autostyle]{csquotes}
\usepackage{enumitem}
\usepackage{cases}
\usepackage{cite}
\usepackage{filecontents}
\theoremstyle{plain}
\newtheorem{thm}{Theorem}[section]

\newtheorem{lem}[thm]{Lemma}
\newtheorem{prop}[thm]{Proposition}

\newtheorem{defn}[thm]{Definition}
\newtheorem{rem}[thm]{Reamrk}

\DeclareMathOperator{\sgn}{sign}

\usepackage{mathtools}

\usepackage{amsmath,amsfonts,amssymb,amsthm,enumerate,multicol}%hyperref,
\usepackage{tikz}
\usepackage{float}
\usepackage{autobreak}

\allowdisplaybreaks

\numberwithin{equation}{section}

\def\ff{u}
\def\vv{U}
\def\kk{\Psi}
\def\bb{f}

\def\mm{M}
\def\tf{\varsigma}

\begin{document}
	\begin{center}
		\Large{\textbf{Mass-Conserving Self-Similar Solutions to Collision-Induced Breakage Equations}}
	\end{center}

	\medskip
	%\centerline{by}
	\medskip
	\centerline{${\text{ ${\text{Ram Gopal~ Jaiswal$^{\dagger}$}}$ and  Ankik Kumar Giri$^{\dagger*}$}}$}\let\thefootnote\relax\footnotetext{$^{*}$Corresponding author. Tel +91-1332-284818 (O);  Fax: +91-1332-273560  \newline{\it{${}$ \hspace{.3cm} Email address: }}ankik.giri@ma.iitr.ac.in}
	\medskip
	{\footnotesize
		%please put the address of the second  and third author

		\centerline{ ${}^{}$  $\dagger$ Department of Mathematics, Indian Institute of Technology Roorkee,}
		%\centerline{Other lines}
		\centerline{Roorkee-247667, Uttarakhand, India}
		
	}

	\bigskip

\begin{quote}
	{\small {\em \bf Abstract.} 
		Existence of mass-conserving self-similar solutions to collision-induced breakage equation is shown for a specific class of homogeneous collision kernels and breakage functions. The proof mainly relies on a dynamical approach and compactness method to constructing mass-conserving stationary solutions for an evolution problem, which induces mass-conserving self-similar solutions to collision-induced breakage equation. Furthermore, we also determine lower and upper bound of the scaling profile.}
\end{quote}

		\vspace{0.5cm}

	
	
	\section{Introduction}\label{sec:intro}
The collision-induced/nonlinear breakage equation describes the dynamics of a system of an infinite number of particles that split into smaller fragments due to collisions. Let $\ff(t,{x})$ denote the density of particles of size $x\in (0,\infty)$ at time $t\ge0$. The evolution of $\ff$ is governed by the integro-partial differential equation provided by

\begin{subequations}\label{main_in_nonlinear_operator}
	\begin{align}
		\partial_{t} \ff (t,{x})&=\mathcal{N}(\ff (t,{x})), \qquad  (t,{x})\in (0,\infty)^2, \label{main}\\
		\ff (0,{x})&=\ff^{\mbox{\rm{\mbox{in}}}}({x})\geq0, \qquad {x}\in (0,\infty),  \label{in}
	\end{align}
	
	where
	\begin{align}
		\mathcal{N} \ff ({x}):=&\frac{1}{2}\int_{{x}}^{\infty}\int_{0}^{{y}}\bb ({x},{y}-{z},{z})\kk({y}-{z},{z})\ff ({y}-{z})\ff ({z})d{z} d {y} \nonumber \\
		&-\int_{0}^{\infty}\kk({x},{y})\ff ({x})\ff ({y})d{y}, \qquad x\in (0,\infty). \label{nonlinear_operator}
	\end{align}
\end{subequations}

In equation \eqref{main_in_nonlinear_operator}, the collision kernel $\kk({x},{y})=\kk({y},{x})\ge0$ for all $({x},{y})\in(0,\infty)^2$ gives the rate at which particles of sizes ${x}$ and ${y}$ collide. The breakage or daughter distribution function $\bb({z},{x},{y})=\bb({z},{y},{x})\ge0$ for all $({x},{y})\in(0,\infty)^2$   determines the distribution of daughter particles of size ${z}\in (0,{x}+{y})$ resulting from the collision of particles of sizes ${x}$ and ${y}$. The first term on the right-hand side of the equation signifies the creation of particles of size ${x}$ due to the collision between a pair of particles of sizes ${y}-{z}$ and ${z}$, where ${y}>{x}$ and ${z}\in(0,{y})$. The second term denotes the reduction in the number of particles of size ${x}$ due to collisions with particles of various sizes in the system.

Assuming that the local mass conservation holds, i.e.,
\begin{equation}
	\int_{0}^{{x}+{y}}{z}\bb ({z},{x},{y})d{z}={x}+{y} ~~\text{and}~~ \bb ({z},{x},{y})=0 ~\text{for}~ {z}\ge{x}+{y}, \label{local_conservation_mass_transfer}
\end{equation}
for all $({x},{y})\in (0,\infty)^2$, is a foundational property in collisional particle breakup modeling. This property ensures that mass is neither created nor destroyed during particles collisions, also leading to the principle of global mass conservation
\begin{equation}
	\int_{0}^{\infty}{x}\ff (t,{x})d{x}=\int_{0}^{\infty}{x}\ff^{\mbox{\rm{\mbox{in}}}}({x})d{x},  \qquad t\geq 0. \label{eq:globalc}
\end{equation}
The equation \eqref{eq:globalc} asserts that the total mass of particles at any time $t$ remains equal to the initial total mass of particles.

Moreover, from equation \eqref{local_conservation_mass_transfer}, combining particles of sizes ${x}$ and ${y}$ through collision cannot yield fragments larger than their sum, ${x} + {y}$. Nevertheless, it is reasonable to expect that collisional fragmentation between ${x}$ and ${y}$ might generate particles larger than ${x}$ and ${y}$ due to the mass transfer from smaller to larger particles.

However, if there exists a function $\bar{f}({z},{x},{y})\ge0$ for $({x},{y})\in(0,\infty)^2$ which determines the distribution of daughter particles of size ${z}\in (0,{x})$ resulting from the collision of particles of sizes ${x}$ and ${y}$ such that

\begin{equation}
	\bb ({z},{x},{y})=\bar{f} ({z},{x},{y})\textbf{1}_{(0,{x})}({z})+\bar{f} ({z},{y},{x})\textbf{1}_{(0,{y})}({z}) \quad \text{for}~~ ({x},{y})\in (0,\infty)^2, \label{no_masstransfer}
\end{equation} 
and 
\begin{equation}
	\int_{0}^{{x}}{z}\bar{f} ({z},{x},{y})d{z}={x} ~\text{and}~~ \bar{f} ({z},{x},{y})=0 ~\text{for}~ {z}\ge{x}, \label{local_conservation_no_mass_transfer}
\end{equation}
then mass transfer during collisions would be absent. 

The integrability of the non-negative function $\bar{f}$ is an important assumption, as it ensures that the number of daughter particles produced during the breakage of a particle of size ${x}$ is finite, which can be mathematically expressed as
\begin{equation}
	\mathrm{N}({x},{y}):=\int_0^{{x}} \bar{f} ({z},{x},{y})d{z} <\infty, ~ ~ ({x},{y})\in (0,\infty)^2. \label{eq:number_of_particles}
\end{equation}
A specific type of breakage function, known as power law breakage, is given by
\begin{equation}
	\bar{f} ({z},{x},{y})=(\nu+2){{z}}^\nu {{x}}^{-\nu-1}\textbf{1}_{(0,{x})}({z}), \label{eq:powerlaw}
\end{equation}
with $\nu\in (-2,0]$ \cite{ziff1987}, and satisfies the integrability assumption \eqref{eq:number_of_particles} when $\nu \in (-1,0]$, but fails to do so when $\nu \in (-2,-1]$.

\medskip

The predictive characteristics of collision-induced breakage equation \eqref{main} poses a central challenge in their application. At first glance, \eqref{main} without mass transfer merely result in the decay of particle sizes towards zero while conserving mass for all times. However, such general information must be more coarse to distinguish specific characteristics among different choices of collision kernel and breakage function. Therefore, a more precise investigation is needed, leading to the development of the dynamical scaling hypothesis. As stated in \cite{piskunov2010,piskunov2012}, the dynamical scaling hypothesis postulates a generalized behaviour of the system. According to this hypothesis, for sufficiently large times, the details of initial particle distribution become insignificant, and the dynamics is governed by the size parameter $e(t)$ of particles at time $t$ and a scaling profile $\eta$ such that the solution $\ff$ to equation \eqref{main_in_nonlinear_operator} assumes the approximate form
\begin{align}
	\ff(t,{x})\sim \frac{1}{e(t)^2}\eta\left(\frac{{x}}{e(t)}\right), \label{self-similar form}
\end{align}
for homogeneous collision kernel $\kk$ and breakage function $\bb$ satisfying
\begin{align}
	\kk(\zeta_1 {x},\zeta_1 {y})=\zeta_1^\lambda\kk({x},{y}),      \quad       \bb(\zeta_2 z,\zeta_2 {x},\zeta_2 {y})=\frac{1}{\zeta_2}\bb({z},{x},{y})    , \label{homogenous_kernels}
\end{align}
for    $(\zeta_1, \zeta_2, {z}, {x}, {y}) \in (0,\infty)^5$ and $\lambda \in (1,2]$.
\medskip

In the specific case where the collision kernel takes the form $\kk(x,y)=a(x)a(y)$, equation \eqref{main_in_nonlinear_operator} can be converted into the spontaneous/linear breakage equation. Over the past few decades, there has been a lot of interest in the linear fragmentation equation \cite{redner1990}, which was first investigated by Filippov \cite{Fil1961}, Kapur \cite{kapur1972}, McGrady and Ziff \cite{ziff1987, ziff1991}  and later studied by functional analytic methods in \cite{banasiak2002, banasiak2004, banasiak2006, bll2019, bt2018, eme2005} and by stochastic approaches in \cite{bert2002, haas2003}, see also the books \cite{bll2019, bert2006} for a more detailed account.

However, the rigorous investigation of the collision-induced breakage equation, as compared to its linear counterpart, is relatively limited. Cheng and Redner  \cite{cheng1988, cheng1990} explored the asymptotic analysis of continuous collision-induced breakage events, focusing on models where the collision of two particles led to various outcomes, such as equal splitting, larger particle splitting, or smaller particle splitting. Krapivsky and Ben-Naim  \cite{krapivsky2003} investigate the dynamics of collision-induced fragmentation, employing its travelling wave behaviour to derive fragment mass distributions. Kostoglou and Karabelas \cite{kostoglou2000, kostoglou2006} studied analytical solutions and asymptotic properties of nonlinear breakage equation for constant, sum, and multiplicative collision kernels. In addition, Ernst and Pagonabarraga \cite{ernst2007} explored the asymptotic behaviour of solutions for the nonlinear breakage equation by transferring it to the linear breakage equation with a different time scale for symmetric breakage with a product kernel $\kk({x},{y}) = ({x}{y})^{q/2}$, where $0\le q \le 2$. Recent studies by Giri and Lauren\c cot \cite{GL12021, GL22021} have made progress in addressing the fundamental question of well-posedness for specific classes of collision kernels and integrable daughter distribution functions without mass transfer and extended in \cite{GJL2024} for non-integrable daughter distribution functions. However, the existence of a self-similar profile in such scenarios, where such a transformation is not available, has remained unexplored.

Primary focus of the present work lies in investigating the existence of mass-conserving self-similar solutions for a specific class of homogeneous collision kernels $\kk$ and homogeneous breakage functions $\bb$ satisfying

\begin{align}
	\kk({x},{y})={x}^{\lambda_1} {y}^{\lambda_2}+{x}^{\lambda_2} {y}^{\lambda_1}, ~~~ ({x},{y})\in (0,\infty)^2,  \label{kernel}
\end{align}

where $k_0 \le {\lambda_1} \le {\lambda_2} \le 1$ for a fix $k_0\in[0,1)$ and $\lambda:={\lambda_1}+{\lambda_2} \in (1,2],$
and
\begin{align}
	\bb({z},{x},{y})=\frac{1}{{x}}{\beta}\bigg(\frac{{z}}{{x}}\bigg) \textbf{1}_{(0,{x})}({z})+\frac{1}{{y}}{\beta}\bigg(\frac{{z}}{{y}}\bigg) \textbf{1}_{(0,{y})}({z}), \qquad ({x},{y})\in(0,\infty)^2, \label{eta_homogeneous_breakage}
\end{align}

subject to the condition

\begin{align}
	\int_0^1 z_*{\beta}(z_*)dz_*=1, \label{eta_local_conservation_no_mass_transfer}
\end{align}
where ${\beta} \in L^1((0,1),z_*dz_*)$, respectively.

Now, let us outline the contents of the paper. In Section 2, some spaces, weak solutions of self-similar profiles along with their alternative form and the main rsults of the paper are summarized. In Section 3, an evolution problem is defined, with the self-similar profile $\eta$ as defined in \eqref{self-similar form} to collision-induced breakage equations. The primary focus in this section is on examining the well-posedness of the corresponding evolution problem and demonstrating important properties of the solutions. Moving to Section 4, a dynamical approach is employed, and a differentiability condition on ${\beta}$ is imposed to establish the existence of a self-similar solution to \eqref{main}. However, the differentiability condition on ${\beta}$ is then relaxed, and the existence of the self-similar solution to \eqref{main} is established. Finally, the properties of the scaling profile are discussed in Section 5.
%%%%%%%%%%%%%
%%%%%%%%%%%%%
\section{Function Space and Main results} \label{function space }
\subsection{Function spaces and notation}

Before stating the main result of the paper, let us introduce the following spaces which are essentially required in subsequent analysis.
\begin{enumerate}[label=(\alph*)]
	\item For $r\in \mathbb{R}$,
	\begin{align*}
		\mathcal{S}_r := L^1((0, \infty), {x}^r d{x})
	\end{align*}
	with a weak topology denoted as $\mathcal{S}_{r,w}$, and its positive cone is denoted as $\mathcal{S}_{r,+}$.
	
	\item
	\begin{align*}
		\mathcal{T}^1 := \left\{\tf\in \mathcal{C}^{0,1}([0,\infty)): \tf(0)=0 \right\}.
	\end{align*}
	
	\item For $k\in (0,1)$,
	\begin{align*}
		\mathcal{T}^k := \left\{\tf\in \mathcal{C}^{0,k}([0,\infty)) \cap L^\infty(0,\infty): \tf(0)=0 \right\}.
	\end{align*}
\end{enumerate}

Also, for $r\in\mathbb{R}$, we define
\begin{align*}
	\mm_{r}(g) := \int_{0}^{\infty } {x}^r g({x}) d{x}, \qquad g \in \mathcal{S}_{r}.
\end{align*}
and
\begin{align*}
	L_{r}^{\lambda_1,\lambda_2}(g) := \mm_{r+{\lambda_1}}(g) \mm_{{\lambda_2}}(g) + \mm_{r+{\lambda_2}}(g) \mm_{{\lambda_1}}(g), \qquad g \in \mathcal{S}_{\lambda_1} \cap \mathcal{S}_{r+{\lambda_2}},
\end{align*}
for ${\lambda_1}$ and ${\lambda_2}$ defined in \eqref{kernel}.

\subsection{Weak solution for self-similar profile}

We insert the self-similar ansatz \eqref{self-similar form} into \eqref{main} and use \eqref{homogenous_kernels} to obtain a pair of uncoupled equations for $e$ and the scaling function $\eta$. The equations are as follows:

\begin{subequations}
	\begin{align}
		e^\prime (t) &= -\omega e(t)^\lambda, \qquad \omega > 0,~ \,~ t > 0, \, ~\lambda \in (1,2], \label{MSP:eq} \\
		\omega[2\eta({x}) + {x}\eta^\prime(x)] &= \mathcal{N}\eta({x}), \qquad {x} > 0, \label{sss_main_omega}
	\end{align}
\end{subequations}

Here, the positive value of $\omega$ is consistent with the decreasing mean size $e(t)$ over time due to {the collisional breakage process without mass transfer}. We infer from \eqref{MSP:eq} that
\begin{align}
	e(t) = [1 + \omega t(\lambda - 1)]^{\frac{1}{(1-\lambda)}}.
\end{align}

Within this context, the objective is to ascertain the existence of two positive real values, namely $(\omega, \rho)$, along with a non-negative function $\eta \in L_1((0,\infty), xdx)$, satisfying the following equations:

\begin{subequations}\label{sss_main_omega_and_mass}
	\begin{align}
		\omega[2\eta({x}) + {x}\eta^\prime(x)]= \mathcal{N}(\eta({x})), \qquad {x} > 0, \label{sss_main_omega}
	\end{align}
	{and}
	\begin{align}
		\int_0^\infty x\eta(x)dx= \rho. \label{sss_mass}
	\end{align}
\end{subequations}

\medskip

It is important to note that equation \eqref{sss_main_omega_and_mass} exhibit an important invariance property which is obtained as follows. If we have a solution of \eqref{sss_main_omega_and_mass} with parameters $(\omega, \rho)$, and we consider the transformation $\eta_{a,b}(x) = a\eta(bx)$, where $(a, b) \in (0, \infty)^2$, then $\eta_{a,b}$ is also a solution to \eqref{sss_main_omega_and_mass} with parameters $(a\omega / b^{\lambda +1}, a\rho b^{-2})$. For simplicity, we can choose $\omega=1/\alpha:=1/(\lambda-1)$ and $\rho=1$, leading to the following equations:

\begin{subequations}\label{sss_main_omega_and_mass=1}
	\begin{align}
		2\eta({x}) + {x}\eta^\prime(x)=\alpha  \mathcal{N}(\eta({x})), \qquad {x} > 0, \label{sss_main_omega with alpha}
	\end{align}
	{and}
	\begin{align}
		\int_0^\infty x\eta(x)dx= \rho. \label{sss_mass=1}
	\end{align}
\end{subequations}
%Additionally, we can express the size parameter $e(t)$ as $e(t) = (1+t)^{-1/\alpha}$ for the chosen value of $\omega$.

\medskip

Now, we define the notion of weak solutions for equation \eqref{sss_main_omega_and_mass=1} that will be utilized throughout the paper.

\begin{defn}\label{defn:weak_solution} Assume that the collision kernel $\kk$ and the breakage function $\bb$ satisfy \eqref{kernel}, \eqref{eta_homogeneous_breakage}, and \eqref{eta_local_conservation_no_mass_transfer}. A self-similar profile to collision-induced breakage equation \eqref{main_in_nonlinear_operator} is a function $\eta \ge 0$ a.e. in $(0,\infty)$  such that
	\begin{align}
		\eta \in \mathcal{S}_{\lambda_1} \cap \mathcal{S}_{1+{\lambda_2}} \qquad \text{and} \qquad  \mm_1(\eta)=1. \label{wf1_sss_mass_conservation}
	\end{align}
	Furthermore, it also satisfies
	\begin{equation}
		\int_{0}^{\infty}[\tf(x)-x\tf^\prime({x})]\eta(x)d{x} =\frac{\alpha}{2}\int_{0}^{\infty}\int_{0}^{\infty}\Upsilon_{\tf}({x},{y})\kk({x},{y})\eta ({x})\eta ({y})d{y} d{x}, \label{weak_formulation_1}
	\end{equation}
	with
	\begin{equation}
		\Upsilon_{\tf}({x},{y}):=\int_{0}^{{x}+{y}}\tf({z})\bb ({z},{x},{y})d{z}-\tf({x})-\tf({y}), \label{upsilon:eq}
	\end{equation}
	for all $\tf \in \mathcal{T}^1$.		
\end{defn}

\medskip

A first consequence of \Cref{defn:weak_solution} is to provide an alternative equation for self-similar profile which also provides a strong integrability than \eqref{eta_local_conservation_no_mass_transfer} on $\beta$.  

\begin{prop}\label{prop: integrability of sss}
	Consider the collision kernel $\kk$ and the breakage function $\bb$ satisfy \eqref{kernel}, \eqref{eta_homogeneous_breakage} and \eqref{eta_local_conservation_no_mass_transfer}. Let $\eta$ be a self-similar profile for the collision-induced breakage equation \eqref{main_in_nonlinear_operator} in the sense of \Cref{defn:weak_solution}. Then, for $x_*\in (0,\infty)$, the following relationship holds
	\begin{align}
		x^2_*\eta(x_*)&={\alpha}\int_{x_*}^\infty \int_0^\infty \int_0^{x_*/ {x}} z_*{\beta}(z_*) {x} \kk({x},{y})\eta({x})\eta({y})dz_* d{y} d{x}, \label{alternate_weak_solution}
	\end{align}
	and $\eta\in \mathcal{C}((0,\infty))$.
	Moreover,
	\begin{align}
		e_{\beta}:=\int_0^1 z_*{\beta}(z_*)|\ln z_*|dz_*<\infty, \label{log estimate on eta}
	\end{align}
	and
	\begin{align}
		L_1^{\lambda_1,\lambda_2}(\eta)=\frac{1}{\alpha e_{\beta}}. \label{L=1/gamma e} 
	\end{align}	
	Furthermore,
	\begin{align}
		\mm_{\lambda_1} (\eta) \ge (\alpha e_{\beta})^{(1-{\lambda_1})/\alpha}, \qquad 
		\text{and} \qquad \mm_{1+{\lambda_2}}(\eta)\le \frac{1}{(\alpha e_{\beta})^{{\lambda_2} / \alpha} }. \label{lower bound on alpha moment}
	\end{align}
	
\end{prop}
\begin{proof}
	
	Let $\xi \in \mathcal{C}_0^\infty((0,\infty))$ and take $\tf({x})={x}\xi({x})$ for ${x}\in (0,\infty)$, we estimate $	\Upsilon_\tf({x},{y})$ with the help of \eqref{eta_homogeneous_breakage} 
	
	\begin{align*}
		\Upsilon_\tf({x},{y})&=\int_0^{{x}+{y}} {z}\xi({z})\bb({z},{x},{y})d{z}-{x}\xi({x})-{y}\xi({y}) \\
		&=\int_0^{x} \frac{{z}}{{x}} \beta\left( \frac{{z}}{{x}}\right) (\xi({z})-\xi({x}))d{z}+	\int_0^{y} \frac{{z}}{{y}} \beta\left(\frac{{z}}{{y}}\right) (\xi({z})-\xi({y}))d{z}\\
		& =-\int_0^{x} \int_{z}^{x}  \xi^\prime(x_*)\frac{{z}}{{x}} \beta\left( \frac{{z}}{{x}}\right)dx_*d{z}-	\int_0^{y}  \int_{z}^{y}  \xi^\prime(x_*) \frac{{z}}{{y}} \beta\left(\frac{{z}}{{y}}\right)dx_*d{z}.
	\end{align*}
	It follows from  \eqref{weak_formulation_1} and Fubini-Tonelli's theorem that
	\begin{align*}		
		\int_0^\infty x_*^2\eta(x_*)\xi^\prime(x_*)dx_*=& \frac{\alpha}{2}\int_0^\infty\int_0^\infty \int_0^{x} \int_0^{x_*}\xi^\prime(x_*)\frac{{z}}{{x}} \beta\left( \frac{{z}}{{x}}\right)\kk({x},{y})\eta({x})\eta({y})d{z} dx_* d{y} d{x}\\
		&+ \frac{\alpha}{2}\int_0^\infty\int_0^\infty \int_0^{y} \int_0^{x_*}\xi^\prime(x_*)\frac{{z}}{{y}} \beta\left( \frac{{z}}{{y}}\right)\kk({x},{y})\eta({x})\eta({y})d{z} dx_* d{y} d{x}\\
		=&  {\alpha}\int_0^\infty  \xi^\prime(x_*) \int_{x_*}^\infty \int_0^\infty \int_0^{x_*}\frac{{z}}{{x}} \beta\left( \frac{{z}}{{x}}\right)\kk({x},{y})\eta({x})\eta({y})d{z} d{y} d{x} dx_*\\
		=& {\alpha}\int_0^\infty  \xi^\prime(x_*) \int_{x_*}^\infty \int_0^\infty \int_0^{x_*/{x}}z_*{\beta}(z_*){x}\kk({x},{y})\eta({x})\eta({y})dz_* d{y} d{x} dx_*.\\
	\end{align*}
	
	Since the above identity being valid for all $\xi\in \mathcal{C}_0^\infty((0,\infty))$, there is a constant $C\in \mathbb{R}$ such that
	\begin{align*}
		x_*^2\eta(x_*)-C={\alpha}\int_{x_*}^\infty \int_0^\infty \int_0^{x_*/{x}}z_*{\beta}(z_*){x}\kk({x},{y})\eta({x})\eta({y})dz_* d{y} d{x}.
		%	     & +\frac{\alpha}{2} \int_0^\infty \int_{x_*}^\infty \int_0^{x_*/{y}}z_*{\beta}(z_*){y}\kk({x},{y})\eta({x})\eta({y})dz_* d{y} d{x} .
	\end{align*}
	
	Finally, we infer from \eqref{eta_local_conservation_no_mass_transfer} that
	\begin{align}
		| x_*^2\eta(x_*)-C|  \le {\alpha}\int_{x_*}^\infty \int_0^\infty {x}\kk({x},{y})\eta({x})\eta({y}) d{y} d{x}. \label{alternate wf 1:eq}
	\end{align}
	Notice that as $x_* \to \infty$, \eqref{alternate wf 1:eq} approaches to zero. This behaviour is consistent with $\eta \in \mathcal{S}_1$  only when $C$ equal to zero. Consequently, this leads to the derivation of the identity \eqref{alternate_weak_solution}. Continuity of $\eta$ concluded with the help of \eqref{alternate wf 1:eq} and integrability characteristics of both $\eta$ and ${\beta}$ listed in Definition \eqref{defn:weak_solution}.
	
	Next, consider $x_0\in (0,\infty)$ and integrate \eqref{alternate_weak_solution} from $x_0$ to $\infty$, we get
	\begin{align*}
		\int_{x_0}^\infty x_*\eta(x_*)dx_*={\alpha}\int_{x_0}^\infty \frac{1}{x_*}\int_{x_*}^\infty \int_0^\infty \int_0^{x_*/ {x}} z_*{\beta}(z_*) {x} \kk({x},{y})\eta({x})\eta({y})dz_*d{y} d{x}  dx_*.
	\end{align*}
	
	Use of Fubini-Tonelli's theorem and \eqref{wf1_sss_mass_conservation}, we get
	
	\begin{align}
		1 \ge \int_{x_0}^\infty x_*\eta(x_*)dx_* =&\alpha \int_{0}^1 \int_0^\infty \int_{x_0}^{x_0/z_*}  z_*{\beta}(z_*) {x} \kk({x},{y})\eta({x})\eta({y})\ln \bigg( \frac{{x}}{x_0}\bigg) d{x} d{y} dz_* \nonumber \\
		&+\alpha \int_{0}^1 \int_0^\infty \int_{x_0/z_*}^{\infty}  z_*{\beta}(z_*) |\ln(z_*)|{x} \kk({x},{y})\eta({x})\eta({y}) d{x} d{y} dz_*.  \label{mass_tail_bound by 1}
	\end{align}
	
	Since ${\beta}$ and $\eta$ are nonnegative, we get
	
	\begin{align*}
		\alpha \int_{0}^1 \int_0^\infty \int_{x_0/z_*}^{\infty}  z_*{\beta}(z_*) |\ln(z_*)|{x} \kk({x},{y})\eta({x})\eta({y}) d{x} d{y} dz_* \le 1.
	\end{align*}
	with the aid \eqref{wf1_sss_mass_conservation}, we can pass the limit $x_0 \to 0$ in the previous inequality and deduce from Fatou's lemma that
	\begin{align*}
		\alpha e_{\beta} L_1^{\lambda_1,\lambda_2}(\eta) \le 1.
	\end{align*}
	Consequently, $e$ is finite and we may then pass the limit  as $x_0 \to 0$ in \eqref{mass_tail_bound by 1} to get $ L_1^{\lambda_1,\lambda_2}(\eta)=1/\alpha e_{\beta}$. From which we conclude that $\mm_{\lambda_1} (\eta)\mm_{1+{\lambda_2}}(\eta) \le 1/\alpha e_{\beta}$ . Since ${\lambda_1} \le 1 \le 1+{\lambda_2}$, we can use H\"older's inequality and deduce \eqref{lower bound on alpha moment}.
\end{proof}

\medskip

\Cref{prop: integrability of sss} reveals a necessary condition: if there exists a self-similar profile for \eqref{main_in_nonlinear_operator} in the sense of Definition \ref{defn:weak_solution}, then it will not exist for any arbitrary ${\beta} \in L^1((0,1)z_*dz_*)$ that satisfies \eqref{eta_local_conservation_no_mass_transfer}, it should also satisfy \eqref{log estimate on eta}. Therefore, we also assume the following assumption $\beta$ which is even stonger than \eqref{log estimate on eta},

\begin{align}
	{\beta} \in L_1((0,1), z_*^{k_0}dz_*), \label{eta_condition_1}
\end{align}
where $k_0$ being defined in \eqref{kernel}. Unfortunately, the existence of self-similar solutions to \eqref{main} has not been covered in this article when $\beta$ satisfies \eqref{log estimate on eta} instead of \eqref{eta_condition_1}, which is shown in the case of linear fragmentation equation \cite{bt2018, Fil1961, bll2019}.

\subsection{Main results} The main results of this paper are following.

\begin{thm}[Existence]\label{thm:existence}
	Let us consider the collision kernel $\kk$ as defined in \eqref{kernel}, with $\bb$ satisfying \eqref{eta_homogeneous_breakage} and ${\beta}$  satisfy \eqref{eta_local_conservation_no_mass_transfer} and \eqref{eta_condition_1}. Then there exists atleast one self-similar profile $\eta$ to \eqref{main_in_nonlinear_operator} according to the  \Cref{defn:weak_solution}, satisfying
	\begin{align}
		\eta \in L_\infty((0,\infty),x^{k_0+1}dx) \cap \bigcap_{k\ge k_0} \mathcal{S}_k, \label{thm3_properties}
	\end{align}
\end{thm}

The proof of \Cref{thm:existence} is based on the dynamical system approach performed in \cite[Section 10.1]{bll2019} to show the existence of self-similar profiles for the linear fragmentation equation.

\begin{rem}
	Note that when $\nu \in (-2, 0]$, the breakage kernel $\bb$ defined as
	\begin{align}
		\bb({z},{x},{y})=(\nu+2)\frac{{z}^\nu}{{x}^{\nu+1}}\textbf{1}_{(0,{x})}({z})+(\nu+2)\frac{{z}^\nu}{{y}^{\nu+1}}\textbf{1}_{(0,{y})}({z})
	\end{align}
	satisfies equations \eqref{eta_homogeneous_breakage}, \eqref{eta_local_conservation_no_mass_transfer}, \eqref{eta_condition_1}, and \eqref{eta_condition_2} with ${\beta}(z_*)=(\nu+2)z_*^{\nu}$. Therefore, \Cref{thm:existence} is applicable to this specific choice of the breakage function.
	%and $\Xi=\frac{|\nu|(\nu+2)}{\nu+k_1}$
\end{rem}

Moreover, the weak solution $\eta$ for \eqref{sss_main_omega_and_mass=1}, as constructed in \Cref{thm:existence}, exhibits the following properties. A similar result is established in \cite[Proposition 10.1.12]{bll2019} for linear fragmentation equation.

\begin{thm}[Strict positiveness and lower bound of the scaling profile]\label{thm:properties_sss}
	Consider the collision kernel $\kk$ and the breakage function $\bb$ satisfying \eqref{kernel}, \eqref{eta_homogeneous_breakage}, \eqref{eta_local_conservation_no_mass_transfer} and {\eqref{eta_condition_1}}. Let $\eta$ be a self-similar profile for the collision-induced breakage equation \eqref{main_in_nonlinear_operator} in the sense of \Cref{defn:weak_solution}. If ${\beta}>0$ a.e in $(0,1)$. Then $\eta$ is strictly positive in $(0,\infty)$ and for all $ x_* \in (0,\infty)$, the lower bounds on the profile $\eta$ is given by
	\begin{align}\label{lower bound on phi}
		r^2e^{\alpha g(r)} \eta(r) \ge  x_*^2e^{\alpha g(x_*)} \eta(x_*), \qquad r \in (x_*,\infty),
	\end{align}
	with
	\begin{align}\label{g(x)}
		g(x_*):=\max\{\mm_{\lambda_1}(\eta), \mm_{\lambda_2}(\eta)\}\bigg(\frac{x_*^{\lambda_1}}{{\lambda_1}}+ \frac{x_*^{\lambda_2}}{{\lambda_2}}\bigg) \quad .
	\end{align}    
\end{thm}

\section{Evolution Problem} \label{evolution problem}
Let us first establish an evolutionary problem by employing scaling variables. To demonstrate the existence of a stationary solution to the evolutionary problem, we will utilize the dynamical approach, as performed in \cite[Section 7.3.2]{bll2019}.

\medskip

Introducing the scaling variables 
\begin{align}
	\tau:=\frac{1}{\alpha}\ln (1+t), \qquad  X:=x(1+t)^{1/\alpha}
\end{align}
and the rescaled function 
\begin{align}
	\vv(\tau,X):=e^{-2\tau} \ff(e^{\alpha \tau}-1,X e^{-\tau}), \qquad (\tau,X)\in[0,\infty)\times (0,\infty). \label{rescaled_function}
\end{align}

Alternatively, we can also write
\begin{align}
	\ff(t,x)=(1+t)^{2/\alpha}\vv\bigg(\frac{1}{\alpha}\ln(1+t),{x}(1+t)^{1/\alpha}\bigg), \qquad (t,x)\in [0,\infty)\times (0,\infty). \label{alternate_rescaled_function} 
\end{align}
Now according to \eqref{main} and \eqref{alternate_rescaled_function},

\begin{subequations}\label{main_in_rescaled}
	\begin{align}
		\partial_\tau \vv(\tau,X)+X&\partial_X \vv(\tau,X)+2\vv(\tau,X) \nonumber\\
		=&\frac{\alpha}{2}\int_X^\infty \int_0^Y \bb(X,Y-Z,Z)\kk(Y,Z)\vv(\tau,Y-Z)\vv(\tau,Z)dZdY \nonumber \\
		& -\alpha \int_0^\infty \kk(X,Y) \vv(\tau, X)\vv(\tau,Y)dY, \qquad (\tau,X)\in (0,\infty)^2. \label{main_rescaled}
	\end{align}
	
	\medskip
	
	The evolution problem we need to solve is represented by \eqref{main_rescaled}, where the stationary solution $\eta$ serves as the solution to \eqref{sss_main_omega_and_mass=1}. In order to establish the existence of a self-similar profile $\eta$ to \eqref{main}, our focus is on demonstrating the existence of stationary solutions for \eqref{main_rescaled}. To fully understand the dynamics of the system described by equation \eqref{main_rescaled}, we include an initial condition
	\begin{align}
		\vv(0,X)=\vv^{in}(X), \qquad X\in(0,\infty). \label{in_rescaled}
	\end{align}
\end{subequations}

The definition, provided in \cite[Definition 2.1]{GJL2024} for weak solutions to the collision-induced breakage equation \eqref{main_in_nonlinear_operator}, can be used to define the notion of  weak solutions for \eqref{main_in_rescaled}. This deduction is possible due to the connection established in \eqref{rescaled_function} and \eqref{alternate_rescaled_function} between the solutions to the \eqref{main_in_nonlinear_operator} and \eqref{main_in_rescaled}.

\begin{defn}\label{defn:rescaled weak solution}
	Let us consider the collision kernel $\kk$ as defined in \eqref{kernel} and  $\bb$ satisfies  \eqref{eta_homogeneous_breakage} and  \eqref{eta_local_conservation_no_mass_transfer}. Let, for fix $k\in [k_0,1)$,  $\vv^{\mbox{\rm{\mbox{in}}}} \in \mathcal{S}_k \cap \mathcal{S}_{1,+}$ and $T\in(0,\infty]$. Then a mass-conserving weak solution to \eqref{main_rescaled}-\eqref{in_rescaled} on $[0,T)$ is a nonnegative function $\vv\in \mathcal{C}([0,T),\mathcal{S}_{k,w}\cap \mathcal{S}_{1,w})$ such that $\vv(0)=\vv^{\mbox{\rm{\mbox{in}}}}$,  
	\begin{equation}
		\int_0^\infty \int_0^\infty (X^k+y^k) \kk(X,Y) \vv(\tau,X)\vv(\tau,Y) <\infty, \quad   \mm_1(\vv(t))=\mm_1(\vv^{\mbox{\rm{\mbox{in}}}}),
	\end{equation}
	and
	\begin{align}
		\int_{0}^{\infty}\tf(X)\vv(t,X)dX =& \int_{0}^{\infty}\tf(X)\vv^{\mbox{\rm{\mbox{in}}}}(X)dX + \int_0^t \int_0^\infty [X\tf^\prime(X)-\tf(X)]\vv(\tau,X)dXd\tau \nonumber \\
		&+\frac{\alpha}{2}\int_{0}^{t}\int_{0}^{\infty}\int_{0}^{\infty}\Upsilon_{\tf}(X,Y)\kk(X,Y)\vv (\tau, X)\vv (\tau, Y)dY dX d\tau, \label{weak_formulation_2}
	\end{align}
	with
	\begin{equation}
		\Upsilon_{\tf}(X,Y):=\int_{0}^{X+Y}\tf(Z)\bb (Z,X,Y)dZ-\tf(X)-\tf(Y), \label{upsilon_rescaled}
	\end{equation}
	for all $\tf \in \mathcal{T}^k \cap \mathcal{T}^1$ and $t\in (0,T)$.
\end{defn}

The analysis conducted in \cite{GJL2024} ensures the existence and uniqueness of weak solution, in the sense of \Cref{defn:rescaled weak solution} to  \eqref{main_in_rescaled}  within the space $\mathcal{S}_{k_*,+}\cap \mathcal{S}_{1+k_*+{\lambda_2}}$ for all $k_*\in (k_0,1)$.  For this purpose we assume that $\beta$ satisfy
\begin{align}
	\Xi:=\int_0^1 z_*^{k_1} |{\beta}^\prime (z_*)|dz_* <\infty, \label{eta_condition_2}
\end{align} 
for some constant $k_1 \in [1,2)$. This estimate is crucial for handling compactness issues within the dynamical system approach. However, Later relaxes this condition by employing approximations of ${\beta}$.

\subsection{Well-posedness of \eqref{main_in_rescaled}}
\begin{prop}\label{wellposedness_evolution_problem}
	Let us consider the collision kernel $\kk$ as defined in \eqref{kernel} and  $\bb$ satisfies  \eqref{eta_homogeneous_breakage} and  \eqref{eta_local_conservation_no_mass_transfer}, while ${\beta}$ satisfies conditions \eqref{eta_condition_1} and \eqref{eta_condition_2}. Consider $k_*\in(k_0,1)$ and an initial condition $\vv^{\mbox{\rm{\mbox{in}}}}$ such that
	\begin{align}
		\vv^{\mbox{\rm{\mbox{in}}}} \in \mathcal{S}_{k_* ,+}\cap \mathcal{S}_{1+k_*+{\lambda_2}} \qquad \text{and} \qquad \mm_1(\vv^{\mbox{\rm{\mbox{in}}}})=1. \label{in_rescaled_assumption}
	\end{align} 
	Then there exists a unique mass-conserving weak solution $\vv$ to  \eqref{main_rescaled}-\eqref{in_rescaled} on $[0,\infty)$. In particular,
	\begin{align}
		\mm_1(\vv(\tau))=\mm_1(\vv^{\mbox{\rm{\mbox{in}}}})=1, \qquad \tau\ge 0.
	\end{align}
\end{prop}

\begin{proof}
	Let $z_* \in(0,1)$ and observe that
	
	\[{\beta}(z_*) = {\beta}(1) - \int_{z_*}^1 {\beta}^\prime(z_*)dz_* \leq [{\beta}(1)+\Xi]z_*^{-k_1}.\]
	
	Next, we fix a value $k_*$ in the range $(k_0,1)$ and define $p_*:= (k_1+k_*-k_0)/k_1$, ensuring $p_* > 1$. For any ${x}\in (0,\infty)$ and $p\in [1, p_*]$, we have
	
	\begin{align*}
		\int_0^{{x}+{y}} {z}^{k_*}\bb({z},{x},{y})^p d{z} &= ({x}^{k_*+1-p}+{y}^{k_*+1-p})\int_0^1 z_*^{k_*} {\beta}(z_*)^p dz_* \\
		&\le [{\beta}(1)+\Xi]^{p-1} ({x}^{k_*+1-p}+{y}^{k_*+1-p})\int_0^1 z_*^{k_0}{\beta}(z_*)dz_*.
	\end{align*}
	
	Therefore, all the required assumptions stated in \cite{GJL2024} are fulfilled and, thus ensures the existence and uniqueness of a mass-conserving solution $\vv \in \mathcal{S}_{k_* ,+}\cap \mathcal{S}_{1+k_*+{\lambda_2}}$ to \eqref{main_in_rescaled} in the sense of Definition \ref{defn:rescaled weak solution}.
\end{proof}

\medskip

Let us establish several estimates that hold for mass-conserving solutions to \eqref{main_rescaled}-\eqref{in_rescaled} under the assumption that the initial conditions satisfy
\begin{align}
	\vv^{\mbox{\rm{\mbox{in}}}} \in \mathcal{S}_{k_0 ,+}\cap \mathcal{S}_{2+{\lambda_2}} \quad \text{and} \quad \mm_1(\vv^{in})=1, \label{in_rescaled_assumption_k_0}
\end{align}
with collisional kernel $\kk$ defined in \eqref{kernel} with ${\lambda_1} \ge k_0$.

Furthermore, we can demonstrate the existence and uniqueness of a mass-conserving solution in this case by utilizing Proposition \ref{wellposedness_evolution_problem}, as $\mathcal{S}_{k_0 ,+}\cap \mathcal{S}_{2+{\lambda_2}} \subset \mathcal{S}_{k_* ,+}\cap \mathcal{S}_{1+k_*+{\lambda_2}}$.

\subsection{Moment estimates}
First, let us define 
\begin{align}
	\Xi_k:=\int_0^1 (Z-Z_*^k){\beta}(Z_*)dZ_*=1-\int_0^1 Z_*^k{\beta}(Z_*)dZ_*, \label{xi_define:eq}
\end{align}
for $k\ge 0$, and note that $\Xi_k>0$ if and only if $k>1$.

\begin{lem} \label{moment_estimates}
	Let us consider the collision kernel $\kk$ as defined in \eqref{kernel} and  $\bb$ satisfies  \eqref{eta_homogeneous_breakage} and  \eqref{eta_local_conservation_no_mass_transfer}, while ${\beta}$ satisfies \eqref{eta_condition_1}. The corresponding solution $\vv$ to  \eqref{main_rescaled}-\eqref{in_rescaled} satisfies the following integrability properties
	
	\begin{enumerate}[label=(\alph*), ref=\thethm (\alph*)]
		\item \label{HME} If $k\ge2$ and $\mm_k(\vv^{\mbox{\rm{\mbox{in}}}})<\infty$, then
		\begin{align}
			\mm_k(\vv(\tau))\le  \max\{M_k(\vv^{\mbox{\rm{\mbox{in}}}}), A_{1,k}\}, \label{lemma_moment estimate_k>1}
		\end{align}
		
		where $A_{1,k}$ is a constant depending only on $\alpha, k,$ and $\Xi_k$.
		
		\item \label{HME (1,2)} If $k\in (1,2)$, then
		\begin{align}
			\mm_k(\vv(\tau))\le  \left( \max\{M_{2+\lambda_2}(\vv^{\mbox{\rm{\mbox{in}}}}), A_{1,2+\lambda_2}\}\right)^{\frac{k-1}{1+\lambda_2}}. \label{lemma_moment estimate_k<2}
		\end{align}

		\item \label{LME} If $k= k_0$, then 
		\begin{align}
			\mm_{k_0}(\vv(\tau)) \le   \max\{M_{k_0}(\vv^{\mbox{\rm{\mbox{in}}}}), A_{2,k_0} \sup_{s\ge0}\{\mm_{1+{\lambda_2}}(\vv(s))^{A_{3,k_0}}\}\},   \label{lemma_moment estimate_k=k_0}
		\end{align}
		where, the constant $A_{2,k}$ depends on ${\lambda_1}, {\lambda_2}, k_0,$ and $\Xi_{k_0}$ and $A_{3,k}$ depends only on ${\lambda_1}, {\lambda_2}$, and $k_0$.
		
	\end{enumerate}
\end{lem}

\begin{proof}
	Consider a test function $\tf(X)=X^k$ for $k\ge k_0$ in the weak formulation \eqref{upsilon_rescaled} and using \eqref{xi_define:eq}, we obtain 
	
	\begin{align*}
		\Upsilon_k (X,Y) &= X^k \int_0^1 Z_*^k{\beta}(Z_*)dZ_* + Y^k \int_0^1 Z_*^k {\beta} (Z_*)dZ_* - X^k - Y^k \\
		&= (X^k +Y^k)\bigg( \int_0^1 Z_*^k{\beta}(Z_*)dZ_*-1\bigg) =-\Xi_k (X^k +Y^k),
	\end{align*} 
	for $(X,Y) \in (0,\infty)^2$. From \eqref{kernel}, \eqref{weak_formulation_2} and above identity for $\Upsilon_k$, we get 
	
	\begin{align}
		\frac{d}{d\tau}\mm_k(\vv(\tau)) =& (k-1)\mm_k(\vv(\tau)) - \alpha \Xi_k \int_0^\infty \int_0^\infty X^k \kk(X,Y)\vv(\tau,X) \vv(\tau,Y)dYdX \nonumber \\
		=& (k-1)\mm_k(\vv(\tau)) \nonumber \\
		&- \alpha \Xi_k[\mm_{k+{\lambda_1}}(\vv (\tau))\mm_{\lambda_2}(\vv (\tau))+\mm_{k+{\lambda_2}}(\vv (\tau))\mm_{\lambda_1}(\vv (\tau))]. \label{lemma_moment estimate_ moment equation}
	\end{align}

	\begin{enumerate}[label=(\alph*), ref=\thethm (\alph*)]
		\item Let us first prove the part \ref{HME}. For this, let $k\ge2$. Then, $\Xi_k>0$. On the one hand  ${\lambda_2}<1<k+{\lambda_1}$, we have following estimates with the help of H\"older's inequality
		\begin{align*}
			\mm_1(\vv(\tau))  &\le \mm_{k+{\lambda_1}}(\vv(\tau)) ^{\frac{1-{\lambda_2}}{k+{\lambda_1}-{\lambda_2}}}\mm_{\lambda_2}(\vv(\tau))^{\frac{k-1+{\lambda_1}}{k+{\lambda_1}-{\lambda_2}}}\\			
			{\mm_1(\vv^{\mbox{\rm{\mbox{in}}}})}^{\frac{k+{\lambda_1}-{\lambda_2} }{k-1+{\lambda_1}}} \mm_{k+{\lambda_1}}(\vv(\tau))^{\frac{k-2+{\lambda_1}+{\lambda_2}}{k-1+{\lambda_1}}}
			& \le \mm_{k+{\lambda_1}}(\vv(\tau))\mm_{\lambda_2}(\vv(\tau))\\
			\mm_{k+{\lambda_1}}(\vv(\tau)) 
			& \le  [\mm_{k+{\lambda_1}}(\vv(\tau))\mm_{\lambda_2}(\vv(\tau))]^{\frac{k-1+{\lambda_1}}{k-2+{\lambda_1}+{\lambda_2}}}.
		\end{align*}

		On the other hand ${\lambda_2}< k <k+{\lambda_1}$, we have 
		\begin{align*}
			\mm_k(\vv(\tau))  & \le \mm_{k+{\lambda_1}}(\vv (\tau))^{\frac{k-{\lambda_2}}{k+{\lambda_1}-{\lambda_2}}} \mm_{\lambda_2}(\vv (\tau))^{\frac{{\lambda_1}}{k+{\lambda_1}-{\lambda_2}}}\\			
			&\le [\mm_{k+{\lambda_1}}(\vv(\tau))M_{\lambda_2}(\vv(\tau))]^{\frac{{\lambda_1}}{k+{\lambda_1}-{\lambda_2}}} \mm_{k+{\lambda_1}}^{\frac{k-{\lambda_1}-{\lambda_2}}{k+{\lambda_1}-{\lambda_2}}}(\vv(\tau))\\			
			&\le [\mm_{k+{\lambda_1}}(\vv(\tau))\mm_{\lambda_2}(\vv(\tau))]^{\frac{{\lambda_1}}{k+{\lambda_1}-{\lambda_2}}} [\mm_{k+{\lambda_1}}(\vv(\tau))\mm_{\lambda_2}(\vv(\tau))]^{\frac{k-1+{\lambda_1}}{k-2+{\lambda_1}+{\lambda_2}} \times \frac{k-{\lambda_1}-{\lambda_2}}{k+{\lambda_1}-{\lambda_2}}}\\			
			&\le [\mm_{k+{\lambda_1}}(\vv(\tau))\mm_{\lambda_2}(\vv(\tau))]^{\frac{k-1}{k-2+{\lambda_1}+{\lambda_2}}}.
		\end{align*}

		Similarly, we obtain
		\begin{align*}
			\mm_k(\vv(\tau)) \le  [\mm_{k+{\lambda_2}}(\vv(\tau))\mm_{\lambda_1}(\vv(\tau))]^{\frac{k-1}{k-2+{\lambda_1}+{\lambda_2}}}.
		\end{align*}
		
		Using above estimates in \eqref{lemma_moment estimate_ moment equation}, we obtain
		\begin{align*}
			\frac{d}{d\tau}\mm_k(\vv(\tau))\le (k-1)\mm_k(\vv(\tau))- 2\alpha \Xi_k \mm_k(\vv(\tau))^{{\frac{k-2+{\lambda_1}+{\lambda_2}}{k-1}}},
		\end{align*}
		for $\tau>0$, from which estimate \eqref{lemma_moment estimate_k>1} follows for $k\ge2$.\\

		\item Thanks to \eqref{in_rescaled_assumption_k_0} and H\"older's inequality to extends the above result for $k\in (1, 2)$ and get \eqref{lemma_moment estimate_k<2}.\\
		
		\item 	Now we move on to the case $k=k_0$, since on the one  hand $({\lambda_1},{\lambda_2})\in [k_0,1]^2$, we can use H\"older's inequality to get the following estimates
		$$\mm_{{\lambda_1}} (\vv(\tau))\leq \mm_{k_0} (\vv(\tau))^{\frac{1-{{\lambda_1}}}{1-k_0}}\mm_1(\vv(\tau))^{\frac{{{\lambda_1}}-k_0}{1-k_0}}\leq \mm_{k_0} (\vv(\tau))^{\frac{1-{{\lambda_1}}}{1-k_0}},$$
		
		$$\mm_{{\lambda_2}} (\vv(\tau))\leq \mm_{k_0} (\vv(\tau))^{\frac{1-{{\lambda_2}}}{1-k_0}}\mm_1(\vv(\tau))^{\frac{{{\lambda_2}}-k_0}{1-k_0}}\leq \mm_{k_0} (\vv(\tau))^{\frac{1-{{\lambda_2}}}{1-k_0}}.$$
		
		On the other hand, either $k_0+{{\lambda_1}} \in [k_0,1]$ and
		$$\mm_{k_0+{{\lambda_1}}}(\vv(\tau))\leq \mm_{k_0}(\vv(\tau))^{\frac{1-k_0-{{\lambda_1}}}{1-k_0}}\mm_1(\vv(\tau))^{\frac{{{\lambda_1}}}{1-k_0}}\leq \mm_{k_0}(\vv(\tau))^{\frac{1-k_0-{{\lambda_1}}}{1-k_0}}$$
		
		or $1 \leq k_0+{{\lambda_1}}\leq 1+{\lambda_2}$ and
		\begin{align*}
			\mm_{k_0+{{\lambda_1}}}(\vv(\tau))&\leq \mm_1(\vv(\tau))^{\frac{1+{\lambda_2}-k_0-{\lambda_1}}{{\lambda_2}}}\mm_{1+{\lambda_2}}(\vv(\tau))^{\frac{k_0+{\lambda_1}-1}{{\lambda_2}}}\le \mm_{1+{\lambda_2}}(\vv(\tau))^{\frac{k_0+{\lambda_1}-1}{{\lambda_2}}}.
		\end{align*}

		Thus, we have
		$$\mm_{k_0+{{\lambda_1}}}(\vv(\tau))\leq \mm_{1+{\lambda_2}}(\vv(\tau)))^{\frac{(k_0+{\lambda_1}-1)_+}{{\lambda_2}}} \mm_{k_0}(\vv(\tau))^{\frac{(1-k_0-{{\lambda_1}})_+}{1-k_0}}.$$

		Similarly,
		$$ \mm_{k_0+{{\lambda_2}}}(\vv(\tau))\leq  \mm_{1+{\lambda_2}}(\vv(\tau))^{\frac{(k_0+{\lambda_2}-1)_+}{{\lambda_2}}} \mm_{k_0}(\vv(\tau))^{\frac{(1-k_0-{{\lambda_2}})_+}{1-k_0}}.$$
		
		Summarizing the above inequalities and use in \eqref{lemma_moment estimate_ moment equation}, we get
		\begin{align*}
			\frac{d}{d\tau}\mm_{k_0}(\vv(\tau)) + (1-k_0&)\mm_{k_0}(\vv(\tau)) \\
			\leq \alpha |\Xi_{k_0}| &\bigg[\mm_{1+{\lambda_2}}(\vv(\tau))^{\frac{(k_0+{\lambda_1}-1)_+}{{\lambda_2}}}
			\mm_{k_0}(\vv(\tau))^{\frac{1-{{\lambda_2}} +(1-k_0-{{\lambda_1}})_+}{1-k_0}} \\
			&+ \mm_{1+{\lambda_2}}(\vv(\tau))^{\frac{(k_0+{\lambda_2}-1)_+}{{\lambda_2}}} 
			\mm_{k_0}(\vv(\tau))^{\frac{1-{{\lambda_1}} +(1-k_0-{{\lambda_2}})_+}{1-k_0}}\bigg].
		\end{align*}

		Since $\lambda={{\lambda_1}}+{{\lambda_2}} \in (1,2]$, then 
		\begin{align*}
			\frac{1-{{\lambda_1}}}{1-k_0}+\frac{(1-k_0-{{\lambda_2}})_+}{1-k_0}&=
			\begin{cases}
				\frac{1-{{\lambda_1}}}{1-k_0}\leq 1& ~\text{if}~k_0+{{\lambda_2}}\geq 1\\
				1+\frac{1-\lambda}{1-k_0}\leq 1& ~\text{if}~k_0+{{\lambda_2}} < 1,
			\end{cases}
		\end{align*}
		
		and, similarly, 
		$$\frac{1-{{\lambda_2}}+(1-k_0-{{\lambda_1}})_+}{1-k_0}\leq 1.$$
		
		Using Young's inequality, we obtain
		\begin{align*}
			\frac{d}{d\tau}\mm_{k_0}(\vv(\tau)) &+ \frac{(1-k_0)}{2}\mm_{k_0}(\vv(\tau))\\
			&\leq  C({\lambda_1}, {\lambda_2}, k_0, \Xi_{k_0}) \sup_{s\ge0}\{\mm_{1+{\lambda_2}}(\vv(s))^{A_{3,k_0}} \}, 
		\end{align*}
		for $\tau>0$, which  provides \eqref{lemma_moment estimate_k=k_0}.
	\end{enumerate}
\end{proof}

%%%%%%%%
%%%%%%%%%
\subsection{Uniform integrability}
We will now proceed to build a weighted $L^1$-estimate on $\partial_X \vv$.

\begin{lem} \label{derivative_moment_estimate}
	Let us consider the collision kernel $\kk$ as defined in \eqref{kernel} and  $\bb$ satisfies  \eqref{eta_homogeneous_breakage} and  \eqref{eta_local_conservation_no_mass_transfer}, while ${\beta}$ satisfies \eqref{eta_condition_1} and \eqref{eta_condition_2}. Consider an initial condition $\vv^{\mbox{\rm{\mbox{in}}}}$ satisfying \eqref{in_rescaled_assumption_k_0} and such that 
	\begin{align*}
		\vv^{\mbox{\rm{\mbox{in}}}} \in \mathcal{S}_1^1((0,\infty), X^{k_1}dX):=\left\lbrace G\in \mathcal{S}_{k_1} : G^\prime\in \mathcal{S}_{k_1} \right\rbrace ,
	\end{align*}
	where the parameter $k_1\in [1,2)$ is introduced in \eqref{eta_condition_2}. Then the corresponding solution $\vv$ to \eqref{main_rescaled}-\eqref{in_rescaled} belong to $\mathcal{S}_1^1((0,\infty), X^{k_1}dX)$  for all times and 
	\begin{align}
		\mm_{k_1} (|\partial_X \vv(\tau)|) \le \max \bigg\{\mm_{k_1}\left\lvert\left(\vv^{\mbox{\rm{\mbox{in}}}}\right)^{\prime}\right \rvert, A_2\sup_{s\ge 0}\{L_{k_1-1}^{\lambda_1,\lambda_2}(\vv(s))\}\bigg\}, \label{uniform bound on derivative}
	\end{align}
	where, the constant $A_2$ depends on ${\lambda_1}, {\lambda_2}, k_1$ and $\Xi$.
\end{lem}

\begin{proof} 
	First note that  $\vv \in \mathcal{S}_{k_1}$  according to Lemma \ref{moment_estimates} as $k_1 \ge 1$. Next we will show that $\partial_X\vv \in \mathcal{S}_{k_1}$. To this end, It follows from \eqref{main_rescaled} that $V=\partial_X\vv$ solves
	\begin{align*}
		\partial_\tau V(\tau,X)=&-X \partial_X V(\tau,X)-3V(\tau,X)\\
		&-\frac{\alpha}{2}\int_X^\infty \int_0^Y \partial_X \bb(X,Y-Z,Z)K(Y-Z,Z)\vv(\tau,Y-Z)\vv(\tau, Z)dZdY\\
		& -\alpha \int_0^\infty \partial_X \kk(X,Y) \vv(\tau,X)\vv(\tau,Y)dY - \alpha \int_0^\infty \kk(X,Y)V(\tau,X)\vv(\tau,Y)dY
	\end{align*}
	for $(\tau,X)\in (0,\infty)^2$. Then, since $\partial_X|V|=q \partial_X V$  with $q:=\sgn(V)$, we use an integration by parts to obtain
	
	\begin{align*}
		\frac{d}{d\tau}& \mm_{k_1}(|V(\tau)|)= \int_0^{\infty} qX^{k_1}\partial_\tau V(\tau,X)dX \le  (k_1-2) \mm_{k_1}(|V(\tau)|)\\
		&+\frac{\alpha}{2}\int_0^\infty qX^{k_1}\int_X^\infty \int_0^Y \partial_X\bb(X,Y-Z,Z)\kk(Y-Z,Z)\vv(\tau,Y-Z)\vv(\tau,Z)dZdYdX\\
		&-\alpha \int_0^\infty qX^{k_1}\int_0^\infty \partial_X \kk(X,Y)\vv(\tau,X)\vv(\tau,Y)dYdX\\		
		\le&  (k_1-2) \mm_{k_1}(|V(\tau)|)\\
		&+\frac{\alpha}{2}\int_0^\infty qX^{k_1}\int_0^\infty \int_0^{Y+Z} \partial_X\bb(X,Y,Z)\kk(Y,Z)\vv(\tau,Y)\vv(\tau,Z)dXdZdY\\
		&-\alpha \int_0^\infty qX^{k_1}\int_0^\infty \partial_X \kk(X,Y)\vv(\tau,X)\vv(\tau,Y)dYdX\\
		\le&  (k_1-2) \mm_{k_1}(|V(\tau)|)\\
		&+ \alpha \Xi \int_0^\infty \int_0^\infty  Y^{k_1-1} \kk(Y,Z)\vv(\tau,Y)\vv(\tau,Z)dZdY\\
		&+\alpha \int_0^\infty X^{k_1}\int_0^\infty \partial_X \kk(X,Y)\vv(\tau,X)\vv(\tau,Y)dYdX\\		
		\le & (k_1-2) \mm_{k_1}(|V(\tau)|)\\
		&+ \alpha \Xi \int_0^\infty \int_0^\infty  X^{k_1-1} \kk(X,Y)\vv(\tau,X)\vv(\tau,Y)dYdX\\
		&+\alpha \int_0^\infty X^{k_1}\int_0^\infty \partial_X \kk(X,Y)\vv(\tau,X)\vv(\tau,Y)dYdX\\
		\le&  (k_1-2) \mm_{k_1}(|V(\tau)|)\\
		&+ \alpha \max\{\Xi, 2{\lambda_2}\}\int_0^\infty \int_0^\infty  [X^{k_1+{\lambda_1}-1}Y^{\lambda_2}+X^{k_1+{\lambda_2}-1}Y^{\lambda_1}]\vv(\tau,X)\vv(\tau,Y)dYdX\\		
		\le & (k_1-2) \mm_{k_1}(|V(\tau)|)+ C({\lambda_1}, {\lambda_2}, \Xi) L_{k_1-1}^{\lambda_1, \lambda_2}(\vv(\tau)).
	\end{align*}
	
	Since ${\lambda_2}\ge {\lambda_1} \ge k_0$ and $k_1+{\lambda_1}-1\ge k_0$, we can conclude that $L_{k_1-1}^{\lambda_1,\lambda_2}(U(\tau))$ is finite for $\tau\ge 0$ and uniformly bounded by Lemma \ref{moment_estimates}, we get 
	\begin{align*}
		\frac{d}{d\tau} \mm_{k_1}(|V(\tau)|)+ (2-k_1)\mm_{k_1}(|V(\tau)|) \le  C({\lambda_1}, {\lambda_2}, \Xi) \sup_{s\ge 0}\{ L_{k_1-1}^{\lambda_1, \lambda_2}(\vv(s))\},
	\end{align*}
	from which \eqref{uniform bound on derivative} follows.
\end{proof}

%%%%%%%%%%%%%
%%%%%%%%%%%%%
\section{Existence by a dynamical approach}
We now provide the proof of \Cref{thm:existence} which relies on a dynamical systems approach.
The proof of \Cref{thm:existence} is not straightforward. Initially, we establish the existence self-similar profile $\eta$ to \eqref{main_in_nonlinear_operator} for daughter distribution functions which satisfy the differentiability estimate \eqref{eta_condition_2} along with the assumption listed in  \Cref{thm:existence}. First let us recall the following theorem from \cite[Proposition~22.13]{amann1990} or \cite[Proof of Theorem~5.2]{gamba2004}.

\begin{thm} \label{thm:DS}
	Let $X$ be a Banach space, $Y$ be a subset of $X$, and $\mathcal{R}:[0,\infty)\times Y\mapsto Y$ be a dynamical system with a positively invariant set $Z\subset Y$ which is a non-empty convex and compact subset of $X$. Then there is $x_0\in Z$ such that $\mathcal{R}(t,x_0)=x_0$ for all $t\ge 0$.
\end{thm}

\medskip

%%%%%%%%%%%%%%%%%%
%%%%%%%%%%%%%%%%%%
\begin{prop}\label{prop:existence with derivative}
	Let us consider the collision kernel $\kk$ as defined in \eqref{kernel} with  $\bb$ satisfies  \eqref{eta_homogeneous_breakage} and ${\beta}$  satisfy  \eqref{eta_local_conservation_no_mass_transfer}, \eqref{eta_condition_1} and \eqref{eta_condition_2}.
	Then there exists a self-similar profile $\eta$ to \eqref{main_in_nonlinear_operator} according to the \Cref{defn:weak_solution} satisfying \eqref{thm3_properties} and 
	\begin{align}
		x^{k_0+1}\eta(x)\le \alpha L_1^{\lambda_1,\lambda_2}(\eta) \bigg(\int_0^1 z_*^{k_0}{\beta}(z_*)dz_*\bigg), \quad x\in(0,\infty). \label{thm3_infty_bound}
	\end{align}
\end{prop}

%%%%%%%%%%%%%%%%
\begin{proof} Let $\rho>0$.\\
	\textbf{Step 1: Invariant set.}
	Utilizing estimates established in \Cref{evolution problem}, we define a subset $Z$ of $\mathcal{S}_{1,+}$ as follows: $\zeta \in Z$ if

	\begin{numcases}{}
		\zeta \in  \mathcal{S}_1^1((0,\infty), X^{k_1}dX) \cap \bigcap_{k\ge k_0} \mathcal{S}_k, \\
		\zeta \ge 0  \quad \text{a.e. in}  \quad (0,\infty),\\		
		\mm_1(\zeta)=1 \label{Z mass},\\
		\mm_k(\zeta)\le \sigma_{1,k}:= \max\{A_{1,k}, \Gamma{(1+k)}\} ~~\text{for} ~~k\ge2, \label{Z k>2} \\
		\mm_k(\zeta)\le \sigma_{2,k}:= \max\{\sigma_{1,k}^{\frac{k-1}{1+\lambda_2}}, \Gamma{(1+k)}\} ~~\text{for} ~~1<k<2, \label{Z 1<k<2}\\
		\mm_{k_0} (\zeta)\le \sigma_{3,k_0}:= \sigma_{1,1+\lambda_2}^{A_{3,k_0}}\max\{A_{2,k_0}, \Gamma{(1+k_0)}\},	\label{Z k0}\\
		\mm_{k_1}\big(|\zeta^\prime(X)|\big) \le \sigma_4:=A_2\left[ \sigma_{3,k_0}^{\frac{1-\lambda_2}{1-k_0}+\frac{2+\lambda_2-\lambda_1-k_1}{1+\lambda_2-k_0}} \sigma_{2,1+\lambda_2}^{\frac{\lambda_1+k_1-k_0-1}{1+\lambda_2-k_0}}  
		+ \sigma_{2,k_0}^{\frac{1-\lambda_1}{1-k_0}+\frac{2-k_1}{1+\lambda_2-k_0}} \sigma_{2,1+\lambda_2}^{\frac{\lambda_2+k_1-k_0-1}{1+\lambda_2-k_0}}  \right] \label{Z k1},
	\end{numcases}	
	where, the constants $A_{1,k}$ for $k\ge 2$, $A_{2,k_0}$ and $A_2$ being defined in
	\Cref{moment_estimates} and \Cref{derivative_moment_estimate}. Clearly, $Z$ is a closed convex subset of $\mathcal{S}_{1,+}$ that contains $x \mapsto e^{-x}$ and is compact in $\mathcal{S}_1$ by \cite[Proposition 7.2.2]{bll2019}, given that $k_0 < 1 \leq k_1 < 2$.
	
	Now, consider $\vv^{\mbox{\rm{\mbox{in}}}} \in Z$ and let $\vv$ be the corresponding solution to \eqref{main_in_rescaled}. It readily follows from \eqref{in_rescaled_assumption_k_0} and Lemma \ref{moment_estimates} that 
	\begin{align}
		\mm_1(\vv(\tau))=1 \quad \text{and} \quad \mm_k(\vv(\tau))\le\sigma_{1,k} , \quad k\ge2 \label{Z_1,k_moment_}
	\end{align}
	and 
	\begin{align}
		\mm_k(\vv(\tau))\le\sigma_{2,k} , \quad 1<k<2  \quad \text{and} \quad	\mm_{k_0} (\vv(\tau))\le \sigma_{3,k_0} \label{Z_k_0_moment}
	\end{align}
	for all $\tau\ge 0$.
	
	Since $k_0\le {\lambda_1}\le {\lambda_2} $ and $k_0\le k_1+{\lambda_1}-1\le 1+\lambda_1$, then by utilizing Lemma \ref{derivative_moment_estimate} \eqref{Z_1,k_moment_}, \eqref{Z_k_0_moment}, and H\"older's inequality, we obtain
	\begin{align}
		\mm_{k_1}\big(|\partial_X\vv(\tau)|\big) \le \sigma_4, 
	\end{align}
	for all $\tau\ge 0$.
	
	In summary, we have demonstrated that
	\begin{align}
		\text{if}  ~~\vv^{\mbox{\rm{\mbox{in}}}}\in Z, ~~\text{then} ~~\vv(\tau)\in Z~~~\text{ for all} ~~~\tau\ge 0. \label{IS:eq}
	\end{align}
	\medskip
	
	\textbf{Step 2: Dynamical system in Z equipped with the weak topology of $\mathcal{S}_1$.}
	Let $\left(\vv_i^{\mbox{\rm{\mbox{in}}}} \right)_{i\ge1} $ denote a sequence of initial conditions in $Z$ that weakly converges in $\mathcal{S}_1$ to $\vv^{\mbox{\rm{\mbox{in}}}} \in Z$. For $i \ge 1$, let $\vv$ be the solution to \eqref{main_rescaled} with initial condition $\vv_i^{\mbox{\rm{\mbox{in}}}}$ provided by \Cref{wellposedness_evolution_problem}. According to \eqref{IS:eq}, the set 
	\begin{align}
		\mathcal{E}:=\left\{\vv_i(\tau) : \tau\ge0, i\ge1\right\}
	\end{align}
	is included in $Z$ and thus is sequentially weakly compact in $\mathcal{S}_1$.
	
	Moreover, from \eqref{kernel}, \eqref{eta_homogeneous_breakage}, \eqref{eta_local_conservation_no_mass_transfer}, \eqref{IS:eq}, and \eqref{main_rescaled}, we readily obtain
	\begin{align*}
		\mm_1(|\partial_{\tau}\vv_i(\tau)|) &\le \alpha \int_0^\infty \int_0^\infty (X+Y)\kk(X,Y)\vv_i(\tau,X)\vv_i(\tau,Y)dY dX\\
		&\le 2\alpha \int_0^\infty \int_0^\infty (X+Y)X^{\lambda_1} Y^{\lambda_2}\vv_i(\tau,X)\vv_i(\tau,Y)dY dX\\
		&\le 2\alpha \left[  \mm_{1+\lambda_1}(U_i) \mm_{\lambda_2}(U_i)+\mm_{1+\lambda_2}(U_i) \mm_{\lambda_1}(U_i)\right] \\
		& \le2\alpha[\sigma_{2,1+\lambda_1}\sigma_{3,\lambda_2}+\sigma_{1,1+\lambda_2}\sigma_{3,\lambda_1}]
	\end{align*}
	for $\tau > 0$ and $j \ge 1$. Consequently, the sequence $(\vv_i)_{i\ge1}$  is equicontinuous at each $\tau > 0$ for the norm-topology of $\mathcal{S}_1$ and thus also for the weak topology of $\mathcal{S}_1$. By utilizing the sequential weak compactness of $\mathcal{E}$ in $\mathcal{S}_1$ along with the equicontinuity of $(\vv_i)_{i\ge1}$ and the variant of the Arzel\'a–Ascoli theorem stated in \cite[Theorem 7.1.16]{bll2019}, we infer that there are a subsequence of $(\vv_i)_{i\ge1}$ (not relabeled) and $\vv \in \mathcal{C} ([0, \infty), \mathcal{S}_{1,w} )$ such that
	\begin{align}
		\vv_{i} \rightarrow \vv ~\text{in} ~~ \mathcal{C} ([0, T ], \mathcal{S}_{1,w} ) \label{SSC:eq}
	\end{align}
	for all $T > 0$. Clearly, $\vv(\tau) \in Z$ for all $\tau \ge 0$. Additionally, the estimates \eqref{IS:eq} and the convergence \eqref{SSC:eq} allow us to apply \eqref{wellposedness_evolution_problem}, concluding that $\vv$ is a unique mass-conserving weak solution to \eqref{main_rescaled} on $[0, \infty)$. The convergence \eqref{SSC:eq} holds true for the entire sequence $(\vv_i)_{i\ge1}$. Thus, we have demonstrated that $\{\vv^{\mbox{\rm{\mbox{in}}}} \mapsto \vv(\tau)\}_{t\ge0}$ is a dynamical system on $Z$ equipped with the weak topology of $\mathcal{S}_1$.
	
	\medskip
	
	\textbf{Step 3: Stationary Solution to \eqref{main_rescaled}.}
	In summary, the map $\{\vv^{\mbox{\rm{\mbox{in}}}} \mapsto \vv(\tau)\}_{\tau\ge0}$ is a dynamical system on $Z$ equipped with the weak topology of $\mathcal{S}_1$. It maps $Z$ into itself, making $Z$ a positively invariant set for the dynamical system. Moreover, $Z$ is a closed convex set which is weakly compact in $\mathcal{S}_1$. The existence of a stationary solution $\eta \in Z$ is then a consequence of \Cref{thm:DS}.
	
	\medskip
	
	\textbf{Step 4: Upper bound on scaling profile.}
	To establish \eqref{thm3_infty_bound}, we revisit \Cref{prop: integrability of sss}, which implies that $\eta$ satisfies \eqref{alternate_weak_solution}, i.e.,
	
	\begin{align*}
		x^2_*\eta(x_*)&={\alpha}\int_{x_*}^\infty \int_0^\infty \int_0^{x_*/ {x}} z_*{\beta}(z_*) {x} \kk({x},{y})\eta({x})\eta({y})dz_* d{y} d{x}  \qquad x_*\in (0,\infty). 
	\end{align*}
	
	Since 
	\begin{align*}
		\int_0^{x_*/{x}}z_*{\beta}(z_*)dz_*\le \left(\frac{x_*}{{x}}\right)^{1-k_0}\int_0^1 z_*^{k_0}{\beta}(z_*),
	\end{align*}
	we deduce that $\eta$ satisfies 
	
	\begin{align*}
		x_*^{k_0+1}\eta(x_*)\le \alpha L_{1}^{\lambda_1,\lambda_2}\int_0^1 z_*^{k_0}{\beta}(z_*),
	\end{align*}
	
	and \eqref{thm3_infty_bound} follows.
	
\end{proof}

%%%%%%%%%%%%%
%%%%%%%%%%%%%%%

We are now in a position to prove \Cref{thm:existence} by eliminating the differentiability requirement on ${\beta}$ defined in \eqref{eta_condition_2} with the help of an approximation procedure.

\begin{proof}[Proof of the \Cref{thm:existence}]
	Firstly, we construct a suitable approximation for ${\beta}$. Let $\phi \in C_c^{\infty}(\mathbb{R})$ be an even, nonnegative function that satisfies
	\begin{align*}
		\int_\mathbb{R} \phi(r)dr=1 \quad \text{and} \quad \text{supp}~\phi \subset (-1,1).
	\end{align*}
	
	For $\delta \in (0,1)$, $r\in \mathbb{R}$, and $z_*\in(0,1)$, we define $\phi_\delta=\delta^{-2}\phi(r \delta^{-2})$, and introduce
	\begin{align*}
		\Phi_\delta:=\int_0^1 z_*\int_\delta^1 \phi_\delta(z_*-z_{**}){\beta}(z_{**})dz_{**}dz_*\\
		{\beta}_\delta(z_*):=\frac{1}{\Phi_\delta}\int_\delta^1\phi_\delta(z_*-z_{**}){\beta}(z_{**})dz_{**}.
	\end{align*}
	
	Then, we have
	\begin{align}
		\lim_{\delta \to 0}\Phi_\delta=1, \quad \lim_{\delta \to 0}\int_0^1 z_*^k |{\beta}(z_*)-{\beta}_\delta(z_*)|dz_*=0, ~~k\ge k_0,  \label{convergence_eta epsilon}
	\end{align}
	and
	\begin{align*}
		{\beta}_\delta^\prime(z_*) \in L^1((0,1),z_*^{k_1}dz_*), \qquad k_1:=1+k_0\in[1,2).
	\end{align*}
	
	Consequently, we conclude that $\kk$ and $\bb_\delta$ defined as
	\begin{align}
		\bb_\delta ({z},{x},{y})=\frac{1}{{x}}\beta_\delta \left( \frac{{z}}{{x}}\right) \textbf{1}_{(0,{x})}({z})+\frac{1}{{y}}\beta_\delta \left( \frac{{z}}{{y}}\right)\textbf{1}_{(0,{y})}({z}),  \label{eq:approx breakage function}
	\end{align}
	satisfy \eqref{no_masstransfer}, \eqref{eta_homogeneous_breakage}, \eqref{eta_local_conservation_no_mass_transfer}, \eqref{eta_condition_1}, and \eqref{eta_condition_2}, for each $\delta \in (0,1)$.

	Now, for each $\delta\in(0,1)$, we can assert the existence of a solution $\eta_\delta$ to \eqref{sss_main_omega_and_mass=1}, where $\bb_\delta$ replaces $\bb$, which is nonnegative and satisfies \eqref{wf1_sss_mass_conservation} as per \Cref{prop:existence with derivative}.
	
	Let us now establish that
	\begin{align*}
		\Omega :=\left\{\eta_\delta : \delta \in (0,1)\right\}
	\end{align*}
	relatively sequentially weakly compact in $\mathcal{S}_1$.
	Using \eqref{convergence_eta epsilon}, we deduce that the family $(\Xi_{k,\delta})$ defined by
	\begin{align}
		\Xi_{k,\delta}:=\int_0^1 (z-z_*^k){\beta_\delta}(z_*)dz_*, \label{xi_define with epsilon:eq}
	\end{align}
	is bounded both from above and below by positive constants depending solely on $k$ and ${\beta}$ for all $k\ge k_0$ and there exist a positive constant $C_{k,\beta}$  depending solely on $k$ and ${\beta}$ for all $k\ge k_0$ such that 
	\begin{align}
		\int_0^1 z_*^{k_0}{\beta}_\delta(z_*)dz_* \le C_{k,\beta},
	\end{align}
	Combining this observation with \eqref{Z k>2}, \eqref{Z 1<k<2}, and \eqref{Z k0}, we infer that the family $\Omega$
	is bounded in $\mathcal{S}_k$ for all $k\ge k_0$. Hence there exist a positive constant $\ell^{\lambda_1,\lambda_2}$ depending on $\lambda_1$ and $\lambda_2$ such that
	\begin{align}
		L_1^{\lambda_1,\lambda_2}(\eta_\delta)\le \ell^{\lambda_1,\lambda_2}.
	\end{align}
	To this end, observe that a further consequence of \eqref{thm3_infty_bound}, that the family $\Omega$ is bounded in $L_\infty((0,\infty),x^{k_0+1}dx)$ and
	
	\begin{align}
		x^{k_0+1}\eta_\delta (x)\le \alpha L_1^{\lambda_1,\lambda_2}(\eta_\delta) \bigg(\int_0^1 z_*^{k_0}{\beta}_\delta(z_*)dz_*\bigg) \le \alpha  \ell^{\lambda_1,\lambda_2} C_{k,\beta} , \quad x\in(0,\infty). \label{eq:infty bouned estimate eta delta}
	\end{align}
	
	Consequently, given a measurable set $E\subset(0, \infty)$, $m \in (0, 1)$ and $M > 1$, we infer from \eqref{eq:infty bouned estimate eta delta} that, for $\delta\in(0,1)$,
	\begin{align*}
		\int_E x \eta_\delta(x)dx  &\le \int_0^m x\eta_\delta(x)dx +\int_m^M \textbf{1}_E(x)x\eta_\delta dx +\int_M^\infty x\eta_\delta (x)dx\\
		&\le \alpha  \ell^{\lambda_1,\lambda_2} C_{k,\beta} \left[ \frac{m^{1-k_0}}{1-k_0}+ \frac{1}{m^{k_0}}|E| + \frac{1}{M^{k_0}}\right].
	\end{align*}
	Therefore the modulus of uniform integrability 
	\begin{equation*}
		\mu\{\Omega;\mathcal{S}_1\} := \lim_{\varepsilon \to 0} \sup\bigg\{ \int_E x\eta_\delta(x) dx  : \eta_\delta\in \Omega, \delta\in(0,1), E\subset(0,\infty), |E|\leq \varepsilon\bigg\}
	\end{equation*}
	of $\Omega$ in $\mathcal{S}_1$ satisfies
	\begin{equation*}
		\mu\{\Omega;\mathcal{S}_1\}\leq \alpha  \ell^{\lambda_1,\lambda_2} C_{k,\beta} \left[ \frac{m^{1-k_0}}{1-k_0}+ \frac{1}{M^{k_0}}\right] 
	\end{equation*}
	for all $0<m<1$ and $ M>1$, we may let $m \to 0$ and  $M \to \infty$ in the previous inequality to obtain that
	\begin{equation}
		\mu\{\Omega;\mathcal{S}_1\}=0.
	\end{equation}
	Also,  arguing as above gives
	\begin{equation}
		\lim_{M \to \infty} \sup_{\delta\in (0,1)} \Big\{\int_{M}^{\infty}x \eta_\delta(x)dx \Big \}=0, \label{eq:cw2}
	\end{equation}
	and we have thus proved the relative sequential weak compactness of $\Omega$ in $\mathcal{S}_1$ by the Dunford-Pettis theorem \cite[Theorem 7.1.3]{bll2019}. Therefore, there exists a sequence $(\delta_j)_{j\ge 1}$ and $\eta\in \mathcal{S}_1$ such that 
	\begin{align}
		\lim_{j \to \infty} \delta_j=0 \quad \text{and} \quad \eta_{\delta_j} \rightharpoonup \eta \quad  \text{in} \quad \mathcal{S}_1 \quad \text{as} \quad  j \to \infty.  \label{convergence phi}
	\end{align}
	With this convergence and the properties of $\Omega$, we deduce that $\eta \in \mathcal{S}_{1,+}$ with $M_1(\eta)=1$ and 
	\begin{align}
		\eta \in L^{\infty}((0,\infty),x^{k_0+1}dx) \cap \bigcap_{k\ge k_0} \mathcal{S}_k.
	\end{align}
	
	\medskip
	%Furthermore, for any $k>k_0$,
	%\begin{align}
	%	\eta_{\delta_j} \rightharpoonup \eta \quad  \text{in} \quad \mathcal{S}_k\quad \text{as} \quad  j \to \infty. \label{convergence phi}
	%\end{align}
	
	We need to verify that $\eta$ is a weak solution to \eqref{sss_main_omega_and_mass=1}. For this purpose, we consider $\tf \in \mathcal{T}^1$ and recall the relation
	\begin{align}
		\int_{0}^{\infty}[x\tf^\prime({x})-\tf(x)]\eta_{\delta_j}(x)d{x} =-\frac{\alpha}{2}\int_{0}^{\infty}\int_{0}^{\infty}\Upsilon_{\tf}^{\delta_j}({x},{y})\kk({x},{y})\eta_{\delta_j} ({x})\eta_{\delta_j} ({y})d{y} d{x}, \label{eq: approx sss wf}
	\end{align}
	where $\Upsilon_{\tf}^{\delta_j}$ is defined as
	\begin{align*}
		\Upsilon_{\tf}^{\delta_j}({x},{y}):=\int_{0}^{{x}+{y}}\tf({z})\bb_{\delta_j} ({z},{x},{y})d{z}-\tf({x})-\tf({y}).
	\end{align*}
	
	Firstly, we estimate $\Upsilon_{\tf}^{\delta_j}$ using \eqref{eq:approx breakage function}:
	\begin{align}
		\Upsilon_{\tf}^{\delta_j}({x},{y})&=\int_{0}^{{x}}\frac{1}{x}\tf({z}){\beta}_{\delta_j}\bigg({\frac{{z}}{{x}}}\bigg)d{z}-\tf({x})
		+\int_{0}^{{y}}\frac{1}{x}\tf({z}){\beta}_{\delta_j}\bigg({\frac{{z}}{{y}}}\bigg)d{z}-\tf({y}) \nonumber\\
		&=\int_{0}^{1}\tf({x} z_*){\beta}_{\delta_j}(z_*)dz_*-\tf({x})
		+\int_{0}^{1}\tf({y} z_*){\beta}_{\delta_j}(z_*)dz_*. \label{eq: estimate of approx upsilon}
	\end{align}
	
	Inserting \eqref{eq: estimate of approx upsilon} into \eqref{eq: approx sss wf}, we obtain
	\begin{align*}
		\int_{0}^{\infty}[\tf(x)-x\tf^\prime({x})]\eta_{\delta_j}(x)d{x} &+\int_0^\infty \int_0^\infty\tf(x)\kk({x},{y})\eta_{\delta_j} ({x})\eta_{\delta_j} ({y})d{y} d{x}\\ &=\alpha\int_{0}^{\infty}\int_{0}^{\infty}\int_{0}^{1}\tf({x} z_*){\beta}_{\delta_j}(z_*)\kk({x},{y})\eta_{\delta_j} ({x})\eta_{\delta_j} ({y})d z_* d{y} d{x}.
	\end{align*}
	
	Introducing
	\begin{align*}
		B_j({x}):=\frac{1}{{x}}\int_0^1 \tf({x} z_*){\beta}_{\delta_j}(z_*)dz_*, \qquad {x} \in (0,\infty),
	\end{align*}
	we use \eqref{eta_local_conservation_no_mass_transfer}, \eqref{convergence_eta epsilon}, and the properties of $\tf$ to show that
	\begin{align*}
		|B_j({x})|\le ||\tf^\prime(x)||_\infty \quad  \text{and} \quad  \lim_{j \to \infty}B_j(x) =\frac{1}{{x}}\int_0^1 \tf({x}, z_*){\beta}(z_*)dz_*.
	\end{align*}
	
	Now, with the help of \eqref{convergence phi}, we can apply \cite[Proposition 7.1.12]{bll2019} to deduce
	\begin{align*}
		&\lim_{j \to \infty} \alpha\int_{0}^{\infty}\int_{0}^{\infty} {x} B_j({x}) \kk({x},{y})\eta_{\delta_j} ({x})\eta_{\delta_j} ({y})d z_* d{y} d{x}\\
		&=\alpha\int_{0}^{\infty}\int_{0}^{\infty}\int_{0}^{1}\tf({x} z_*){\beta}_{\delta_j}(z_*)\kk({x},{y})\eta_{\delta_j} ({x})\eta_{\delta_j} ({y})d z_* d{y} d{x}.
	\end{align*}
	
	Therefore, the proof of \Cref{thm:existence} is complete.
\end{proof}
%%%%%%%%%%%%%%%%
%%%%%%%%%%%%%%%%

%%%%%%%%%%%%%%%
%%%%%%%%%%%%%%%%
\section{Strict positiveness and lower bound of the scaling profile}
In this section, we consider the collision kernel $\kk$ and the breakage function $\bb$ that satisfy \eqref{kernel}, \eqref{eta_homogeneous_breakage}, \eqref{eta_local_conservation_no_mass_transfer}, and {\eqref{eta_condition_1}}. Let $\eta$ be a mass-conserving self-similar profile for the collision-induced breakage equation \eqref{main_in_nonlinear_operator}, following the definition provided in \Cref{defn:weak_solution}.

As previously mentioned, the proof of \Cref{thm:properties_sss} is adapted from \cite[Proposition 10.1.12]{bll2019}, and the initial step is to establish that $\eta^\prime \in \mathcal{S}_2$.

\begin{lem}\label{lem:negativemoments}
	Consider a mass-conserving self-similar profile $\eta$ for the collision-induced breakage equation \eqref{main_in_nonlinear_operator}, as constructed in \Cref{thm:existence}. Then
	\begin{equation}
		\eta^\prime \in \mathcal{S}_2.
	\end{equation}
\end{lem}
%%%%%%%%%%%%%%%%

\begin{proof}
	From \eqref{eta_homogeneous_breakage} and \eqref{alternate_weak_solution}, we deduce that
	\begin{align*}
		\left(x_*^2\eta(x_*)\right)^\prime=& -\alpha \int_0^\infty x_* \kk(x_*,{y})\eta(x_*)\eta({y})d{y}\\
		&+ \alpha \int_{x_*}^\infty \int_0^\infty \frac{x_*}{{x}} \beta\left( \frac{x_*}{{x}}\right)  \kk({x},{y}) \eta({x}) \eta({y})d{y} d{x}.
	\end{align*}
	This expression holds for almost every $x_*\in(0,\infty)$. Utilizing \eqref{wf1_sss_mass_conservation}, \eqref{eta_local_conservation_no_mass_transfer}, and the Fubini-Tonelli theorem, we can conclude that
	\begin{align*}
		\int_0^\infty \left \lvert \left(x_*^2\eta(x_*)\right)^\prime\right \rvert dx_* \le 2\alpha L_1^{\lambda_1,\lambda_2}(\eta)<\infty.
	\end{align*}
	Hence, we obtain the inequality
	\begin{align*}
		\int_0^\infty {x_*^2 }\left \lvert 	\eta^\prime(x_*) \right\rvert dx_* &=	\int_0^\infty \left\lvert \left(x_*^2\eta(x_*)\right)^\prime\right \rvert dx_*+ 2\int_0^\infty x_* \eta(x_*)dx_*\\
		&\le 2(\alpha L_1^{\lambda_1,\lambda_2}(\eta)+ 1)<\infty.
	\end{align*}
\end{proof}
\medskip

%%%%%%%%%%%
%%%%%%%%%%%

Now, let us establish the lower bound on the scaling profile $\eta$ as stated in Theorem \ref{thm:properties_sss}.
\begin{proof}[Proof of Theorem \ref{thm:properties_sss}]
	By utilizing \eqref{eta_homogeneous_breakage}, \eqref{eta_local_conservation_no_mass_transfer}, and \eqref{alternate_weak_solution}, we obtain
	\begin{align}\label{thm4_1}
		\left(x_*^2 e^{\alpha g(x_*)}\right)^\prime= \alpha e^{\alpha g(x_*)}\int_{x_*}^\infty \int_0^\infty \frac{x_*}{{x}} \beta\left( \frac{x_*}{{x}}\right)\eta({x})\eta({y}) \ge 0,
	\end{align}
	for almost every $x_*\in (0,\infty)$, where $g(x)$ is defined by \eqref{g(x)}. Given that $\eta$ is not identically zero due to \eqref{wf1_sss_mass_conservation}, there exists $x_0 \in (0,\infty)$ such that $\eta(x_0) > 0$. By examining \eqref{thm4_1}, we observe that
	\begin{align}\label{thm4_2}
		x_*^2e^{\alpha g(x_*)} \eta(x_*) \ge  x_0^2e^{\alpha g(x_0)} \eta(x_0) >0,
	\end{align}
	whenever $x_* \in (x_0 , \infty)$.
	
	Moreover, assuming, for contradiction, that there exists $x_1 \in (0, x_0)$ such that $\eta(x_1) = 0$ and $\eta(x_*) > 0$ for $x_* \in (x_1,x_0]$, we can derive from \eqref{alternate_weak_solution} the following result
	\begin{align*}
		\int_{x_1}^\infty \int_0^\infty \int_0^{x_1/ {x}} z_*{\beta}(z_*) {x} \kk({x},{y})\eta({x})\eta({y})dz_* d{y} d{x} = 0.
	\end{align*}
	Considering the positivity of $\kk$ and ${\beta}$, this identity implies that $\eta(x_*) = 0$ for $x_* \in (x_1, \infty)$, contradicting the definition of $x_1$. Hence, we conclude that $\eta$ is positive in $(0, x_0)$ and, consequently, in $(0, \infty)$ based on \eqref{thm4_2}. Finally, integrating \eqref{thm4_1} over the interval $(x_*,r)$, we obtain \eqref{lower bound on phi}.
\end{proof}

\textbf{Acknowledgments.} 
	This research received financial support from the Indo-French Centre for Applied Mathematics (MA/IFCAM/19/58) as part of the project ``Collision-induced breakage and coagulation: dynamics and numerics." The authors would also like to acknowledge Council of Scientific \& Industrial Research (CSIR), India for providing a PhD fellowship to RGJ through Grant 09/143(0996)/2019-EMR-I. Additionally, authors extend their appreciation to Prof. Philippe Lauren\c cot for valuable discussions.
%%%%%%%%%%%%%%%%%%

	\bibliography{Refs.bib}
\bibliographystyle{abbrv}

\end{document}